%% file: bqp_ras.tex
\documentclass[onefignum,onetabnum]{siamart171218}

\input{ex_shared}
\usepackage{adjustbox}
\usepackage{rotating}
\newsiamthm{assumption}{Assumption}

\renewcommand{\P}{\mathbb{P}}

\ifpdf
\hypersetup{
	pdftitle={A Random Active Set Method for Strictly Convex Quadratic Problem with Simple Bounds},
	pdfauthor={Ran Gu, and Bing Gao}
}
\fi

\begin{document}
	\maketitle
	
	\begin{abstract}
		Active set method aims to find the correct active set of the optimal solution and it is a powerful method for solving strictly convex quadratic problem with bound constraints. To guarantee the finite step convergence, the existing active set methods all need strict conditions or some additional strategies, which greatly affect the efficiency of the algorithm. In this paper, we propose a random active set method which introduces randomness in the update of active set. We prove that it can converge in finite iterations with probability one without any conditions on the problem or any additional strategies. Numerical results show that the algorithm obtains the correct active set within a few iterations, and compared with the existing methods, it has better robustness and efficiency.
	\end{abstract}
	
	
	
	\section{Introduction}\label{sec:1}
	In this paper, we consider Strictly Convex Quadratic Problem (SCQP) with simple bound constraints. Without loss of generality, we let the simple bound constraints be nonnegative constraints. Then the problem can be written as
	\begin{equation}\label{qp}
	\begin{array}{cl}
	\min\limits_{x\in \mathbb{R}^n} & \frac{1}{2}x^TQx+g^Tx\\
	\text{s.t.} & x\geq 0,
	\end{array}
	\end{equation}
	where $g\in \mathbb R^n$ and $Q\in \mathbb R^{n\times n}$ is a symmetric positive definite matrix. Although, \cref{qp} is a simple convex optimization problem, it is widely used and has very important research significance. For example, it often appears in subproblems of some optimization algorithms for solving complex optimization problems, such as augmented Lagrangian function method for solving bound constraint problems \cite{conn1991globally,birgin2010global}, block coordinate descent method for solving nonnegative matrix factorization \cite{wang2012nonnegative}, sequential quadratic programming method for solving nonlinear problems \cite{gill2005snopt}, etc. In those problems,  \cref{qp} needs to be solved repeatedly,  so it is
	very important to solve \cref{qp} efficiently in order to greatly improve the efficiency of these algorithms.
	
	The common methods to solve problem \cref{qp} are interior point method, gradient projection method and active set method. Interior point algorithm minimizes a series of obstacle functions with parameters by Newton method to find the feasible searching direction, which is the dominant part of its computation cost, and then searches for the next interior point \cite{vanderbei1999loqo,wright1997primal,koh2007interior}. Different from interior point method, gradient projection method, whose search direction is the projection of the current gradient, allows the iterative points to move at the boundary of the feasible region \cite{conn2013lancelot,dai2005projected,more1991solution,rosen1960gradient}. Both of the above methods are feasible methods. Active set method was also a feasible method in the early days, but with people's attention to guessing the correct active set, a series of infeasible active set methods came into being. Next, we introduce it in detail.
	
	Since problem \cref{qp} is convex, solving the problem is equivalent to solving its Karush-Kuhn-Tucker (KKT) system, that is,
	\begin{equation}\label{kkt}
	\begin{array}{c}
	Qx+g-s=0,\\
	x^Ts=0,\\
	x\geq 0, s\geq 0,
	\end{array}
	\end{equation}
	where $s$ is the variable of Lagrange multipliers. We use $A\subset \{1,\ldots,n\}$ to represent the active set and $ I=\{1,\ldots,n\}\backslash A $ to denote the inactive set. That is to say,
	for the minimizer $x^*$ of (\ref{kkt}), we have $x^*_A=0$ and $x^*_I>0$. By the complementary slackness, we obtain $s^*_I=0$ and $s^*_A\geq 0$. Here, the notation $x_A$ represents the components of $x$ indexed by $A$. If $Q$ is a matrix and $A$ and $B$ are index sets, then $Q_{A,B}$ is denoted as the sub-matrix of $Q$ by $Q_{A,B} = (q_{ij}),\ i\in A,j\in B$. Using these notations, we rewrite the KKT system as
	\begin{equation}\label{kktas}
	\begin{array}{c}
	Q_{I,I}x^*_I+g_I=0,\\
	Q_{A,I}x^*_I+g_A-s_A=0,\\
	x^*_I\geq 0, s_A\geq 0.
	\end{array}
	\end{equation}
	
	It can be found that in \cref{kktas} the optimal solution can be directly solved by solving the first linear equations if the correct inactive set of the minimizer $ I $ is known in advance. Generally, the correct inactive set or active set can not be obtained at the beginning, thus active set methods iteratively adjust the estimated active set to find the correct one at the end. In those studies, Fletcher \cite{fletcher1971general} proposed an active set method to solve linear constrained convex quadratic programming. It contains an outer loop and an inner loop. In the inner loop, the algorithm iteratively reduces the cardinality of $I$ by one until the solved $x_I$ is positive. If the solved $s_A$ is nonnegative, then the KKT point is obtained. Otherwise, in the outer loop, it chooses an index $ a $ from $A$ with a corresponding negative $s_a$ and move it to $I$. In the algorithm, the iterative point remains feasible and the objective function decreases monotonically, which guarantees the finite step convergence of the algorithm. From the point of view of solving system, the KKT system \cref{kkt} itself is a linear complementarity problem, and the above active set method consists with the principal pivot method for solving such a linear complementarity problem. For more information on principal pivot methods, one can refer to \cite{judice1994block}.
	
	People are not satisfied with changing only one index of the active set at each iteration, because although the rank one update of Cholesky decomposition can speed up the calculation on solving linear equations, the overall computation is still large. People attempted to allow multiple simultaneous exchanges in each iteration. Different from the classical active set methods, for rapid adaptation of the active set estimate, a block principal pivoting algorithm \cite{judice1994block} is proposed. In \cite{kunisch2003infeasible}, Kunisch and Rendl proposed their KR algorithm, where in each iteration, KR computes $x_I$ and $s_A$ from \cref{kktas}, and exchanges all the indexes of negative elements in primal and dual variables. In their numerical experiments, the algorithm only needs several iterations to find the optimal solution. The finite step convergence was also proved under certain assumptions on the condition number of $Q$. Hinterm{\"u}ller, Ito, and Kunisch \cite{hintermuller2002primal} showed that this primal-dual active set approach can be rephrased as a semi-smooth Newton method.
	
	The assumptions of the condition number $Q$ mentioned above are not satisfied in many practical problems, so KR method often generates cycles, and the finite step convergence is invalid, which means that it can not be used safely. For example, in \cite{curtis2015globally} a 3-dimensional example showed the failure of KR method, where the iterative points construct a cycle; a large number of random tests also showed the failures of KR method in \cite{hungerlander2015feasible}. To conquer the potential cycle, Curtis, Han and Robinson \cite{curtis2015globally} introduced an uncertain set of indexes and a reduced subproblem in their algorithm, by which the finite step convergence is guaranteed without any assumptions. In \cite{hungerlander2015feasible}, Hungerl{\"a}nder and Rendl modified KR method by introducing a feasiblization approach to guarantee the monotonic descent of the objective function. After that, the relationship between semi-smooth Newton method and KR method was deeply explored, and an infeasible active set method combined with line search was proposed to guarantee the finite step convergence \cite{hungerlander2016infeasible}. In 2018, Huyer and Neumaier extended feasible active set approaches to indefinite quadratic programs \cite{huyer2018minq8}.
	
	The rapid adjustment of active set and the finite step convergence make it possible for active set method to beat other methods. It has many advantages over other methods, such as finding the exact numerical solution and being insensitive to initialization. Furthermore, strict complementarity is not required. In some cases, active set methods can outperform the primal-dual interior point methods in computing time, because active set methods solve the linear equations only on primal variables, while the interior point methods solve the entire system. However, for the above-mentioned active set methods, the rapid adjustment of active set and the finite step convergence are contradictory, because in order to ensure finite step convergence, they all utilize additional strategies, which increases the computational cost.
	
	In this paper, our goal is to make active set method converge within a finite number of steps on the basis of effectively adjusting the active set, without adding redundant calculations. The method we are going to propose is based on a simple idea. Intuitively, since the essence of making finite step convergence untenable is the "cycle", a natural way to break the cycle is to introduce randomness. In detail, we can randomly select the exchanging index set so that the randomness will help to avoid the cycle and provide more possibilities for the success of the algorithm. This idea constitutes the basic framework of our new method, which is called Random Active Set method(RAS). This is the first work to combine randomness and active set method. Although it seems to be a small change in the original framework, it is very effective. We will show the efficiency of RAS by some numerical experiments. It is worth mentioning that our new algorithm performs very well in a large number of randomly generated examples. Besides, we also prove the finite step convergence of RAS in the sense of probability.
	
	The rest of this paper is organized as follows: In \cref{sec:2}, we propose a generic RAS framework and show the convergence result of the new method. In \cref{sec:3}, we give the details of some probability settings in RAS. In \cref{sec:4}, a large amount of numerical instances are tested to illustrate the performance of our proposed algorithm. Finally, in \cref{sec:5}, we summarize the main contributions of this paper and put a perspective on some outstanding challenges.

	
	\section{Random active set method}\label{sec:2}
	
	As introduced in \cref{sec:1}, all the indexes $ \{1,\ldots,n\} $ are divided into inactive set $I$ and active set $A$. Then, according to the sign of the elements of  $x_I$ and $s_A$, which are obtained by solving the first two equations in \cref{kktas}, $I$ can be divided into $Ip=\{i\in I|x_i>0\}$ and $Im=\{i\in I|x_i\leq 0\}$, and $A$ can be divided into $Ap=\{j\in A|s_j\geq 0\}$ and $Am=\{j\in A|s_j<0\}$. The above-mentioned active set methods all adopt deterministic rules, to exchange part of $Im$ and $Am$ in each iteration, and obtain new $I$ and $A$ for iteration until $Im$ and $Am$ become empty sets, which also implies that the KKT system \cref{kktas} is solved.
	
	The basic idea of RAS is to use randomness to select the indexes to be exchanged, rather than deterministic rules. To describe the algorithm, we need a symbol to show that some elements are randomly selected from a set. Here we define
	$$
	B=rand(A,p)
	$$
	to be a subset of $A$, where $A$ is an index set, $p=[p_1,p_2,\ldots,p_{|A|}]$ is a vector with the same size as $A$, and the $i$-th element of $A$ exists in $B$ with probability $p_i$. Here $p$ can also be a scalar for convenience, if all $p_i$ takes the same value.
	
	\begin{figure}
		\centering
		\includegraphics[scale=0.5]{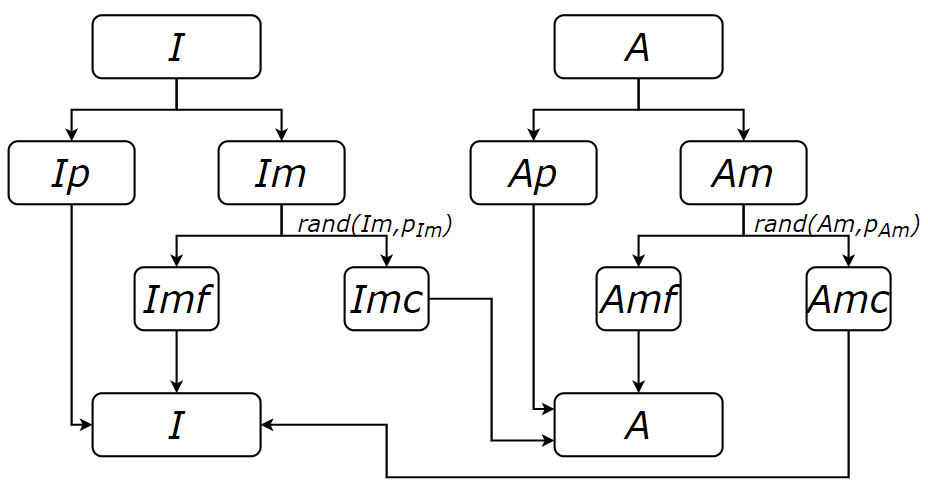}\\
		\caption{Flow chart of generic random active set method}\label{fig:gras}
	\end{figure}
	
	\cref{fig:gras} shows the flow chart of generic RAS. In RAS, each index in $Im$ and $Am$ has a certain probability of leaving the current set. In $Im$, the set of indexes that is removed is marked as $Imc$, and the set of remaining indexes is marked as $Imf$. Correspondingly in $Am$, the set of indexes to leave is marked as $Amc$, and the set of remaining indexes is marked as $Amf$. Then the new $I$ and $A$ are obtained by $I = Ip\cup Imf\cup Amc$ and $A = Ap\cup Amf\cup Imc$. The setting rule of each probability are not specified here, so we call this method generic RAS, and its algorithm framework is shown in \cref{gras}. Note that we have given the upper and lower bounds of each probability, that is, each probability is selected from $[\sigma, 1-\sigma]$. This makes each index in $Im$ and $Am$, whether exchanged or reserved, have a probability greater than zero, which plays an important role in proving the finite step convergence in the sense of probability.
	
	\begin{algorithm}
		\caption{A generic random active set method}\label{gras}
		\begin{enumerate}
			\item[Step 0:] Give a constant $\sigma$ such that $0<\sigma\leq 0.5$, the initial guess on active set $A$ and $I=\{1,\ldots,n\}\setminus A$, and a tolerance on dual nonnegativity violation $tol\geq 0$.
			\item[Step 1:] Solve the first two lines in \cref{kktas} and classify the infeasible indexes as follows: \\
			$x_I = -Q_{I,I}^{-1}g_I$,\\
			$s_A = Q_{A,I}x_I + g_A$,\\
			$Im = \{i\in I|x_i \leq 0\}$, $Am = \{j\in A| s_j < -tol\}$, $Ip=I\setminus Im$, $Ap=A\setminus Am$.
			\item[Step 2:] If $Im\cup Am=\emptyset$, stop. \\
			Otherwise, take $p_{Im}\in \mathbb R^{|Im|}$ and
			$p_{Am}\in \mathbb R^{|Am|}$ with each element belonging to $[\sigma, 1-\sigma]$. Then let
			$Imc=rand(Im,p_{Im})$, $Imf=Im\setminus Imc$, $Amc=rand(Am,p_{Am})$, $Amf=Am\setminus Amc$.
			\item[Step 3:]
			Update $I = Ip\cup Imf\cup Amc$, $A = Ap\cup Amf\cup Imc$. Go to Step 1.
		\end{enumerate}
	\end{algorithm}
	
	Based on the above facts, we notice that RAS can produce the same index set change as Fletcher's active set method \cite{fletcher1971general} with a certain probability. Next, we analyze the finite step convergence of RAS in the sense of probability by using the finite step convergence property of Fletcher's active set method. In detail, \cref{lem:finite1} indicates that RAS is consistent with Fletcher's active set method under a certain probability. \cref{lem:finite2} describes a characteristic of RAS, that is, it can terminate within a certain number of iterations  from any iteration with a positive probability. Then the final conclusion is given in \cref{thm} that RAS can terminate in finite iterations with probability 1.
	
	\begin{lemma}\label{lem:finite1}
		For any $\sigma\in (0,0.5]$, there exists $\delta_1>0$, satisfying the following statement:\\
		Take any $x$ as the iterative point of Fletcher's active set method \cite{fletcher1971general} and the corresponding inactive set is $I$, the complement of $I$ is $A$.
		Then the probability that the new $I$ and $A$ generated by \cref{gras} are the same as those generated by Fletcher's active set method is at least $\delta_1 $.
	\end{lemma}
	\begin{proof}
		First of all, let's review the Fletcher's active set method introduced in \cref{sec:1}, in which the update rule consists of two-level loops. In the inner loop, the number of elements of $I$ is reduced by one in each iteration, while in the outer loop, the number of elements of $I$ is increased by one. We discuss these two cases as follows.
		
		When $Im$ is not empty, the iterative point $x_I$ will move straightly to $x_I^*=-Q_{II}^{-1}g_I$ until reaching the feasible boundary. That is to say, there is one index changing from $Im$ to $A$ in Fletcher's active set method. In RAS, the probability of moving $i$ from $Im$ to $A$ is
		$$
		\left(\prod\limits_{j\in Im, j\neq i}(1-p_j)\right)p_i\left(\prod\limits_{j\in Am}(1-p_j)\right)\geq \sigma^n
		$$
		
		When $Im$ is empty and $Am$ is not empty, there is one index changing from $Am$ to $I$ in Fletcher's active set method. In RAS, the probability of moving $i$ from $Am$ to $I$ is
		$$
		\left(\prod\limits_{j\in Am, j\neq i}(1-p_j)\right)p_i\left(\prod\limits_{j\in Im}(1-p_j)\right)\geq \sigma^n
		$$
		
		When $Im$ and $Am$ are empty, both Fletcher's active set method and RAS stop.
		
		To sum up, the lemma holds for $\delta_1 = \sigma^n$.
	\end{proof}
	
	\begin{lemma}\label{lem:finite2}
		For any $\sigma\in (0,0.5]$, there exist constants $\delta_2>0$ and $\bar K$, such that the probability of \cref{gras} terminating in the subsequent $\bar K$ iterations from any iterative point is at least $\delta_2$.
	\end{lemma}
	\begin{proof}
		For any given iteration of \cref{gras}, suppose that its inactive set is $I$, and a corresponding feasible point $x$ can be constructed. We implement Fletcher's active set method for the initial point $x$ until the algorithm terminates. Then in the $k$-th iteration, the algorithm generates an inactive set $I^{(k)}$. We regard the series of $I^{(k)}$ as a sequence, and define its length as $K_I$, which is finite because of the finite step convergence property of Fletcher's active set method. Since the number of combinations of $I$ is finite, we record the maximum value of $K_I$ as $\bar K$, that is,
		$$
		\bar K= \max\limits_{I\subseteq \{1,\ldots,n\}} K_I.
		$$
		Combined with \cref{lem:finite1}, we have
		$$
		\begin{array}{l}
		\quad\P(\text{RAS terminates in the subsequent $\bar K$ iterations})\\
		\geq \P(\text{RAS generates the same index set sequence as Fletcher's active set method})\\
		\geq  \delta_1^{\bar K}
		\end{array}
		$$
		Then the lemma holds for $\delta_2=\delta_1^{\bar K}$.
	\end{proof}
	
	\begin{theorem}\label{thm}
		\cref{gras} can stop in finite iterations with probability 1.
	\end{theorem}
	\begin{proof}
		We give proof by contradiction. Assume that the probability of infinite step convergence of \cref{gras} is greater than 0, that is,
		$$
		\inf\limits_N \{\P(\text{the number of iterations of \cref{gras}} \geq N)\}=\alpha>0.
		$$
		Then, there exists $N_0$ such that
		$$
		\P(\text{the number of iterations of \cref{gras}} \geq N_0)< \frac{\alpha}{1-\delta_2},
		$$
		where $\delta_2$ is the constant obtained from \cref{lem:finite2}.
		
		Combined with \cref{lem:finite2}, we have
		$$
		\begin{aligned}
		&\P(\text{the number of iterations of \cref{gras}} \geq N_0+\bar K)\\
		\leq &\P(\text{the number of iterations of \cref{gras}} \geq N_0) (1-\delta_2)\\
		< &\alpha(1-\delta_2)^{-1}(1-\delta_2)\\
		=&\alpha
		\end{aligned}
		$$
		This contradicts with the fact that $\alpha$ is the infimum.
	\end{proof}

	\section{Probability setting}
	\label{sec:3}
	In this section, we will give the specific probability setting based on generic RAS.  In Step 2 of generic RAS, the probability specified by each index can be arbitrarily selected within a given range. If the probability update rule is designed for each index separately, the number of hyperparameters is too large, which is not good for algorithm design. On the contrary, if the probabilities of all indexes are set to a unified fixed number, the algorithm will lose its efficiency. Therefore, we hope to find some characteristics, divide the whole index set into multiple categories, and formulate unified update rule for each category.
	
	Intuitively, the two index sets $I$ and $A$ are asymmetric. The set $I$ participates in solving the linear equations in KKT system \cref{kktas}, but $A$ does not. Therefore, the probabilities on $Im$  and $Am$  should be different. Here we use a two-dimensional example to illustrate this finding. See the following proposition.
	
	\begin{proposition}\label{prop:1}
		Let $Q\in \mathbb R^{2\times 2}$ be a random positive definite matrix with $Q_{12}$ mean-zero symmetrically distributed. Suppose that $g\in \mathbb R^2$ is a random vector independent with $Q$, whose random variables $g_1$, $g_2$  satisfying that the value of the joint probability distribution function on any zero-centered circle is fixed, and $\P(g_1=g_2=0)<1$. Suppose for case 1, $I=Im=\{1,2\},A=\emptyset$, in the next iteration, $I^+=\emptyset, A^+=\{1,2\}$; For case 2, $I=\emptyset,A=Am=\{1,2\}$, in the next iteration, $I^+=\{1,2\}, A^+=\emptyset$. Denote the number of infeasible indexes in case 1 and case 2 in the next iteration as $N1$ and $N2$ respectively, then $\mathbb E(N1)\leq \mathbb E(N2)$ and the equality holds if and only if $\P(Q_{12}=0) = 1$.
	\end{proposition}
	
	\begin{proof}
	According to $Qx+g=s$, it follows that $x_1=(-Q_{22}g_1+Q_{12}g_2)/|Q|,x_2=(Q_{12}g_1-Q_{11}g_2)/|Q|$ when $I=\{1,2\}$, and $s_1=g_1,s_2=g_2$ when $I=\emptyset$. By $Q\succ 0$, we have the condition
	\begin{equation}\label{pre_con0}
		\left\{
		\begin{aligned}
		Q_{11}Q_{22}-Q_{12}^2>0\\
		Q_{11}>0,\, Q_{22}>0
		\end{aligned}
		\right.	
	\end{equation}
	\newline
	{\bf Case 1:}
	
	By the assumption on case 1, we have
		\begin{equation}\label{pre_con1}
		\left\{
		\begin{aligned}
		-Q_{22}g_1+Q_{12}g_2\leq 0\\
		Q_{12}g_1-Q_{11}g_2\leq 0
		\end{aligned}
		\right.	
		\end{equation}
	For the next iteration, $A^+=\{1,2\}$,
		$$\left\{
		\begin{aligned}
		s_1^+ & = g_1 \\
		s_2^+ & = g_2
		\end{aligned}
		\right.
		$$
		then we obtain that
			\begin{center}
		\setlength\tabcolsep{3pt}
		\begin{tabular}{cll}
		if $Q_{12}>0$,& $N1=0$, & \\
		if $Q_{12}<0$,& $N1=0$, & when $g_1\geq 0,g_2\geq 0$,\\
		&$N1=1$, & when $g_2Q_{12}/Q_{22}\leq g_1<0$ or $g_1Q_{12}/Q_{11}\leq g_2<0$,\\
		if $Q_{12}=0$, & $N1=0$. &
		\end{tabular}
		\end{center}
	To sum up, if $\P(Q_{12}=0) = 1$, the expectation $\mathbb E(N1)=0$. Otherwise, $\P(Q_{12}<0) >0$, then $\mathbb E(N1)$ can be calculated by
    $$
		\begin{aligned}
		\mathbb E(N1) &= \P(Q_{12}< 0) \P(g_2Q_{12}/Q_{22}\leq g_1<0\\
		&\quad \text{ or }g_1Q_{12}/Q_{11}\leq g_2<0|\cref{pre_con1},\cref{pre_con0},Q_{12}<0)\\
		&= \P(Q_{12}<0) [\P(g_2Q_{12}/Q_{22}\leq g_1<0|\cref{pre_con1},\cref{pre_con0},Q_{12}<0)\\
		&\quad +\P(g_1Q_{12}/Q_{11}\leq g_2<0|\cref{pre_con1},\cref{pre_con0},Q_{12}<0)]\\
		&= \P(Q_{12}<0)[\P(g_2Q_{12}/Q_{22}\leq g_1<0,\cref{pre_con1}|\cref{pre_con0},Q_{12}<0)\\
		&\quad +\P(g_1Q_{12}/Q_{11}\leq g_2<0,\cref{pre_con1}|\cref{pre_con0},Q_{12}<0)]/\P(\cref{pre_con1}|\cref{pre_con0},Q_{12}<0)\\
		&= \P(Q_{12}<0)[\P(g_2Q_{12}/Q_{22}\leq g_1<0|\cref{pre_con0},Q_{12}<0)\\
		&\quad +\P(g_1Q_{12}/Q_{11}\leq g_2<0|\cref{pre_con0},Q_{12}<0)]/\P(\cref{pre_con1}|\cref{pre_con0},Q_{12}<0)\\
		\end{aligned}
    $$
    Let
    $$
    \begin{aligned}
     E_1&=\P(g_2Q_{12}/Q_{22}\leq g_1<0|\cref{pre_con0},Q_{12}<0),\\
     E_2&=\P(g_1Q_{12}/Q_{11}\leq g_2<0|\cref{pre_con0},Q_{12}<0),\\
     E_3&=\P(g_1>0,g_2>0|\cref{pre_con0},Q_{12}<0),
    \end{aligned}
    $$
   then  we have
    $$
    \mathbb E(N1)=\P(Q_{12}<0)(E_1+E_2)/(E_1+E_2+E_3).
    $$
    Here, $E_1> 0,E_2>0,E_3>0$ when $\P(Q_{12}=0) < 1$, which can be derived from the assumption that the function value of the joint probability distribution function of $g_1$ and $g_2$ on any zero-centered circle is fixed and  $\P(g_1=g_2=0)<1$.
    \newline
	{\bf Case 2:}
	
	By the assumption on case 2, we have
		\begin{equation}\label{pre_con2}
		\left\{
		\begin{aligned}
		g_1< 0\\
		g_2< 0
		\end{aligned}
		\right.	
		\end{equation}
	For the next iteration, $I^+=\{1,2\}$,
		$$\left\{
		\begin{aligned}
		sign(x_1^+) & = sign(-Q_{22}g_1+Q_{12}g_2) \\
		sign(x_2^+) & = sign(Q_{12}g_1-Q_{11}g_2)
		\end{aligned}
		\right.
		$$
		then we obtain that
	    \begin{center}
		\setlength\tabcolsep{3pt}
		\begin{tabular}{cll}
		if $Q_{12}>0$,& $N2=0$, & when $g_2Q_{11}/Q_{12}<g_1<g_2Q_{12}/Q_{22}$,\\
		&$N2=1$, & when $g_2Q_{12}/Q_{22}\leq g_1<0$ or $g_1Q_{12}/Q_{11}\leq g_2<0$,\\
		if $Q_{12}<0$,& $N2=0$, & \\
		if $Q_{12}=0$, & $N2=0$. &
		\end{tabular}
		\end{center}
	To sum up, if $\P(Q_{12}=0) = 1$, the expectation $\mathbb E(N2)=0$. Otherwise, $\P(Q_{12}>0) >0$, then $\mathbb E(N2)$ can be calculated by
    $$
		\begin{aligned}
		\mathbb E(N2) &= \P(Q_{12}> 0) \P(g_2Q_{12}/Q_{22}\leq g_1<0\\
		&\quad \text{ or }g_1Q_{12}/Q_{11}\leq g_2<0|\cref{pre_con2},\cref{pre_con0},Q_{12}>0)\\
		&= \P(Q_{12}>0) [\P(g_2Q_{12}/Q_{22}\leq g_1<0|\cref{pre_con2},\cref{pre_con0},Q_{12}>0)\\
		&\quad +\P(g_1Q_{12}/Q_{11}\leq g_2<0|\cref{pre_con2},\cref{pre_con0},Q_{12}>0)]\\
		&= \P(Q_{12}>0)[\P(g_2Q_{12}/Q_{22}\leq g_1<0,\cref{pre_con2}|\cref{pre_con0},Q_{12}>0)\\
		&\quad +\P(g_1Q_{12}/Q_{11}\leq g_2<0,\cref{pre_con2}|\cref{pre_con0},Q_{12}>0)]/P(\cref{pre_con2}|\cref{pre_con0},Q_{12}>0)\\
		&= \P(Q_{12}>0)[\P(g_2Q_{12}/Q_{22}\leq g_1<0|\cref{pre_con0},Q_{12}>0)\\
		&\quad +\P(g_1Q_{12}/Q_{11}\leq g_2<0|\cref{pre_con0},Q_{12}>0)]/P(\cref{pre_con2}|\cref{pre_con0},Q_{12}>0)\\
		&=\P(Q_{12}>0)(E_1+E_2)/E_3
		\end{aligned}
    $$
    Here, the last equality is obtained by the symmetry of $g_1$, $g_2$ and $Q_{12}$.

    If $\P(Q_{12}>0)=\P(Q_{12}<0)=0$, we obtain $\mathbb E(N1)=\mathbb E(N2)=0$. Otherwise, it follows that
    $$
    \frac{\mathbb E(N1)}{\mathbb E(N2)}=\frac{E_3}{E_1+E_2+E_3}< 1.
    $$
    Therefore, the proposition is true.
	\end{proof}
	
	Although \cref{prop:1} is only a special two-dimensional example, it tells us that the probabilities for the index sets $Im$ and $Am$ should be set separately. The assumption on the distributions of $Q$ and $g$ is common, where normal distribution and uniform distribution are included, and the two cases are fair to $A$ and $I$. Since the ultimate goal of RAS is to eliminate all infeasible indexes, the probability of moving indexes out of $Im $ should be higher than that of moving indexes out of $Am $, which is inspired by \cref{prop:1}. Besides, moving indexes from $Im$ to $A$ can reduce the size of linear equations that to be solved in KKT system, which is time saving.
	
	In addition, the set to which an infeasible index belongs in the previous iteration also affects their probability settings. For example, in the current iteration we have two infeasible indexes, one of which is feasible in the last iteration and the other is infeasible.
	Intuitively, we tend to move the index which is infeasible in two consecutive iterations with higher probability, and move the other index with lower probability.
	
	Therefore, when setting the changing probability of all infeasible indexes in the current iteration, we can trace back to the set to which they belong in previous few iterations. Different probability update rules can be designed for different combinations of current and historical index sets. Here, in order to reduce the number of hyperparameters in RAS, we only consider the latest previous iteration and give each combination a fixed changing probability.
	
	We first name the combinations of current and historical index sets in \cref{tab:sets}. For current iteration, we only consider the infeasible sets $Im$ and $Am$, and for previous iteration, we consider all four sets $Ip$, $Im$, $Ap$ and $Am$.  Since the indexes in $Ip$ and $Ap$ are still in $I$ and $A$ respectively in the next iteration, there are two non-existent combinations in the table, which are represented by '-'. Thus, we have six categories in total, which are named as NImp0, NImf, NImc, NAmp0, NAmf and NAmc. We specify the changing probabilities corresponding to the above six categories as $p_1$, $p_2$, \ldots, $p_6$, respectively. Then the flow chart of RAS is displayed in \cref{fig:ras}, and the  pseudo code is listed in \cref{ras}.
	
	\begin{table}[htb]
		\centering
		\begin{tabular}{|c|c|c|c|c|c|}
			\hline
			\multicolumn{2}{|c|}{}       & \multicolumn{4}{c|}{previous sets} \\
			\cline{3-6}
			\multicolumn{2}{|c|}{}       & Ip & Im & Ap & Am \\
			\hline
			current & Im & NImp0 & NImf & - & NImc \\
			\cline{2-6}
			sets & Am & - & NAmc & NAmp0 & NAmf \\
			\hline
		\end{tabular}%
		\caption{Combinations of the sets in the current and previous iteration} \label{tab:sets}
	\end{table}
	
	\begin{figure}[htb]
		\centering
		\includegraphics[scale=0.8]{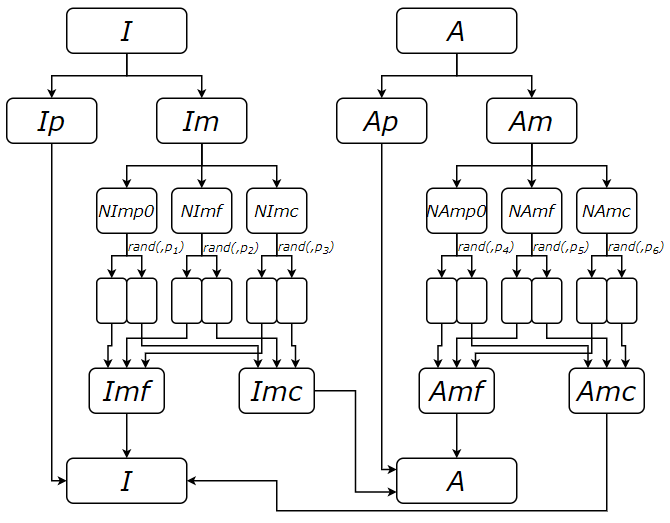}\\
		\caption{Flow chart of Random Active Set method}\label{fig:ras}
	\end{figure}
	
	\begin{algorithm}
		\caption{Random active set method}\label{ras}
		\begin{enumerate}
			\item[Step 0:] Give the initial guess on active set $A$ and $I=\{1,\ldots,n\}\setminus A$, probability $0<p_i<1$, $i=1,\ldots,6$, and the tolerance on dual nonnegativity violation $tol\geq 0$. Set $Ip0=Ap0=Imc=Amc=\emptyset$ and $Imf=Amf=\{1,\ldots,n\}$.
			\item[Step 1:] Solve the first two lines in \cref{kktas} and classify the infeasible indexes as follows \\
			$x=0$, $x_I = -Q_{I,I}^{-1}g_I$,\\
			$s=0$, $s_A = Q_{A,I}x_I + g_A$,\\
			$Im = \{i|x_i \leq 0\}$, $Ip = I \setminus Im$, $Am = \{i| s_i < -tol\}$, $Ap = A \setminus Am$,\\
			$NImp0= Im\cap Ip0$, $ NImf = Im\cap Imf$, $NImc = Im\cap Amc$,\\
			$NAmp0 = Am\cap Ap0$, $NAmf = Am\cap Amf$, $NAmc = Am\cap Imc$.
			\item[Step 2:] If $Im\cup Am=\emptyset$, stop. \\
			Otherwise, select a new active set randomly according to the following rules,\\
			$Imc = rand(NImp0,p_1)\cup rand(NImf,p_2)\cup rand(NImc,p_3)$,\\
			$Amc = rand(NAmp0,p_4)\cup rand(NAmf,p_5)\cup rand(NAmc,p_6)$,\\
			$Imf = Im\setminus Imc$, $ Amf = Am\setminus Amc$,\\
			$I = Ip\cup Imf\cup Amc$, $A = \{1,\ldots,n\}\setminus I$,
			\item[Step 3:] $Ip0=Ip$, $Ap0=Ap$, go to Step 1.
		\end{enumerate}
	\end{algorithm}
	
	\cref{ras} is based on the framework of \cref{gras}, setting the changing probability in a special way. As long as $0<p_i<1$($i=1,\ldots,6$) is satisfied, the convergence result of this algorithm is the same as \cref{gras}, which can be guaranteed by \cref{thm}.
	
	It should be noted that in the implementation of RAS, if the changing sets are empty, i.e., the active set is the same for two consecutive iterations, the KKT equation does not need to be solved again, and this iteration is not counted in the iteration count. It just changes the category of some indexes and regenerates random numbers.
	
	In order to obtain the optimal value of the probability $p_i (i=1,\ldots,6)$, we design a new optimization problem, in which the variables are $0.01\leq p_i\leq 0.99(i=1,\ldots,6)$ and the objective function is the summation of the number of iterations of RAS running on several randomly generated problems. We regard this problem as a black box optimization problem and solve it by surrogate method \cite{regis2007stochastic,wang2014general}. Due to the randomness of the problem, we repeatedly solve it 50 times and take the mean value of the solution as the final values of $p_i(i=1,\ldots,6)$, which are
	$$p_1=0.5,p_2=0.98,p_3=0.98,p_4=0.01,p_5=0.93,p_6=0.94.$$
	These $p_i$ are also more in line with our intuitive understanding. For example, $p_1>p_4,p_2>p_5,p_3>p_6$, which is consistent with \cref{prop:1}.
	In addition, $p_1 < p_2$ and $p_4 < p_5$, which are also corresponding to the intuition that the probability of moving an index that is infeasible in two consecutive iterations should be greater than the probability of moving an index that is infeasible only in the current iteration. So far, we have completed all the details of RAS.
	
	\section{Numerical tests}
	\label{sec:4}
	In this section, we report some numerical results to demonstrate the performance of our proposed algorithm RAS. We compare our RAS algorithm with KR algorithm, MKR algorithm and interior point method on three benchmark sets. To comprehensively test the performance of the algorithms, the benchmark sets cover different dimensions, different sparsity, different types of spatial rotation and different condition numbers. All experiments were conducted on an Intel Core i7-7500U CPU with 2.70 GHz and 2.90 GHz and 16 GB RAM. The MATLAB code of RAS is available on http://github.com/guran1214/Random-Active-Set-Method.
	
	The ways to generate $Q$ affect the complexity of the algorithms. Based on the performance of KR and MKR in each benchmark, we divide the difficulty into three levels with three cases correspondingly. In the easy case, KR can converge in a few steps. In the medium case, KR needs more iterations to converge or fails to converge in some instances. In both cases above, $Q$ is a randomly generated sparse matrix, and the corresponding MATLAB codes for generating instances and the codes of KR and MKR algorithms are obtained from the link in \cite{hungerlander2015feasible}. In the hard case, we introduce a class of dense matrices, which makes the failure rate of KR higher and the number of iterations of MKR larger.
	
	In the following tables, `time' is the running time of the algorithm (in seconds); `solve' is the number of linear equations to be solved; `avgI' is the average size of linear equations to be solved; `fail' is the number of instances that KR can not solve within 200 iterations. Here, since KR may fall into a cycle, we add a limit of 200 iterations to it, which is consistent with the setting in \cite{hungerlander2015feasible}. There are multiple nesting in MKR, so we use the number of solving linear equations instead of the number of iterations.
	
	\subsection{Easy case}
	Firstly, we study the performance of RAS in the easy case. The problem generator is same as the first problem in Section 6.1 in \cite{hungerlander2015feasible}. We let $Q$ to be a sparse matrix, which is a random symmetric multi-diagonal matrix plus a small multiple of the identity matrix, making it  strictly positive definite. The specific MATLAB code is as follows
	$$p = sprandn(n, n, .1) + speye(n);p=tril(triu(p,-100));Q0=p*p';Q=Q0+\epsilon I;$$
	Consistent with \cite{hungerlander2015feasible}, we fix the dimension as $n = 2000$, the tolerance on dual nonnegativity violation of all algorithms as $tol=1e-8$, and test the influence of condition number on the algorithms by changing $\epsilon=1,1e-5,1e-10,1e-14$. In fact, we have learned from \cite{hungerlander2015feasible} that for this benchmark set, the influence of condition number is not significant, and both KR and MKR can converge within a few steps even with the highest condition number.
	
	In \cref{tab:t1}, we add RAS to the comparison. From the table, we can see that in the easy case, all the three algorithms perform well and take little time to solve the linear equations. In terms of the number of solving linear equations, KR wins the first place because it does not add any safeguard strategies and does not limit the changes of infeasible indexes, so it solves fewer linear equations than MKR and RAS. RAS has probabilistic restrictions on the change of infeasible indexes, but compared to MKR, it has less number of solving linear equations. In terms of computational time, RAS is close to KR, because the size of linear equations solved is smaller (see the last three lines in \cref{tab:t1}). KR and RAS are both faster than MKR.
	
	In \cref{fig:easy}, we can see the changes in the number of infeasible indexes in an instance of the benchmark. Since MKR is a nested algorithm, we only record on its outermost layer. The number of infeasible indexes in all algorithms decreases monotonically for this easy case.
	
	Therefore, in this benchmark set, RAS is similar to KR and MKR in that it is little affected by the condition number. In terms of numerical performance, the number of linear equations to be solved in RAS is larger than that of KR but smaller than that of MKR, and the total computational time of RAS is close to KR and less than MKR.
	
	\begin{figure}
		\centering
		\includegraphics[scale=0.8]{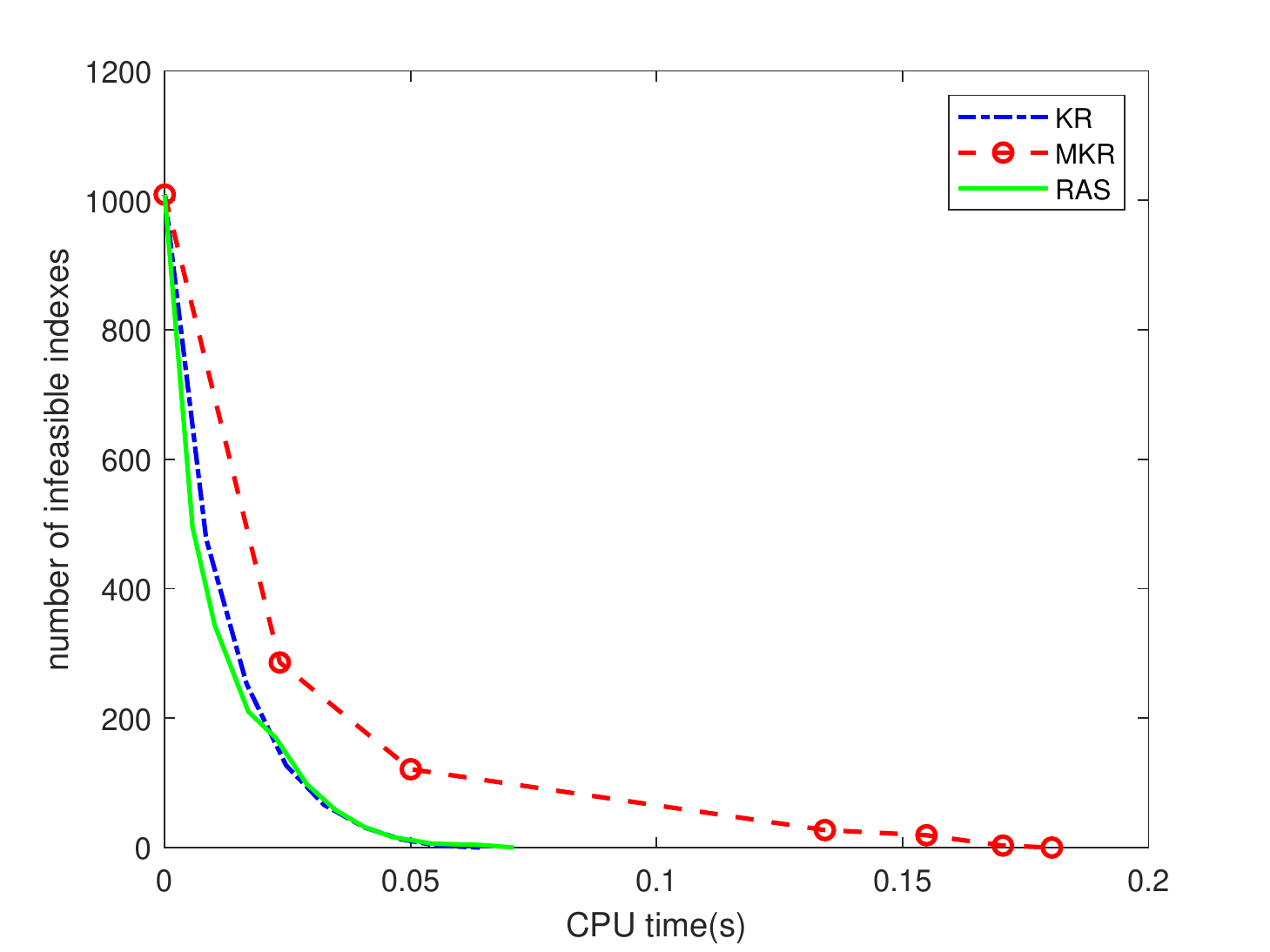}\\
		\caption{The number of infeasible indexes in KR, MKR and RAS changes with CPU time. (Easy case with $n = 2000$ and $\epsilon=1e-10$)}\label{fig:easy}
	\end{figure}
	
	\begin{table}
		\setlength\tabcolsep{3pt}
		\begin{tabular}{|r|r|r|r|r|r|r|r|r|r|r|r|}
			\hline
			cond & $\lg$($\epsilon$) & \multicolumn{4}{c}{kr} & \multicolumn{3}{|c}{mkr} & \multicolumn{3}{|c|}{ras} \\
			&       & time & solve & avgI & fail & time & solve & avgI & time & solve & avgI \\
			$\approx$1.00E+03 & 0 & 0.04 & 5.9 & 821.01 & 0 & 0.07 & 9.3 & 842.86 & 0.05 & 9.1 & 847.04 \\
			$\approx$1.00E+08 & -5 & 0.07 & 9.2 & 879.47 & 0 & 0.10 & 14.2 & 853.97 & 0.07 & 12.1 & 868.24 \\
			$\approx$1.00E+13 & -10 & 0.07 & 9.1 & 878.64 & 0 & 0.10 & 14.2 & 853.58 & 0.07 & 12.0 & 867.70 \\
			$\approx$1.00E+17 & -14 & 0.06 & 9.1 & 878.64 & 0 & 0.11 & 14.2 & 853.58 & 0.07 & 12.0 & 867.70 \\
			\hline
		\end{tabular}%
		\caption{Comparison of RAS, KR and MKR; $n = 2000$ and initial active set is $\{1,\ldots,n\}$. Each line is the average over 10 trials.} \label{tab:t1}
	\end{table}

	\subsection{Medium case}
	
	Next, we consider the random problem generated by the style given in \cite{curtis2015globally} which is also the second type problem in Section 6.1 in \cite{hungerlander2015feasible}. Specifically, we generate $ Q $ through the sprandsym routine of MATLAB, where $ Q $ is a sparse matrix with different options for its sparsity, dimension and condition number. We present these results in \cref{tab:t2}. Here, the built-in routine "quadprog" of MATLAB 2020a is also included in the comparison, in which the interior-point-convex method and the default termination criterion are adopted. For quadprog, we use
	$$error=\frac{\text{objective function value} - \text{optimal function value}}{|\text{optimal function value}|} $$
	to express the inaccuracy of its solution. Since the algorithm has a predictor-corrector, the linear equations is solved twice in each iteration. Besides, there is a 1e-10 tolerance on dual nonnegativity violation for KR, MKR and RAS.
	
	These examples are more difficult than the previous ones in the last subsection. Once the condition number becomes larger, the number of iterations of KR also becomes larger, and KR fails on a subset of the instances. Thus the average in KR given in \cref{tab:t2} considers only the instances that are successfully solved. In general, MKR and RAS in this benchmark set are faster than the interior point method in quadprog, which is consistent with our prediction, because interior point method requires solving a full-size linear equations, while active set method only needs to solve a small one.
	
	Next, we compare the performance of RAS and MKR in \cref{tab:t2}. The tests show that the number of solving linear equations in both MKR and RAS increases with the increase of the $ Q $'s condition number. In addition, we find that the number of solving linear equations also increases with the increase of $Q$'s density, except for a few cases in RAS. Compared with MKR, RAS changes less with different condition numbers and different densities. We can see that when $n = 5000$ and $den = 0.1$, the number of solving equations in MKR more than 400, while in RAS it is only 42. Besides, the computational time of RAS is a little more than one-seventh of that of MKR. In less extreme cases, the number of iterations and time consumed by RAS are usually about one-third of that of MKR, which fully demonstrates the superiority of RAS.
	
	In \cref{fig:medium}, we show the changes in the number of infeasible indexes for active set methods in an instance. KR does not converge. The number of infeasible indexes in RAS decreases non-monotonically. Although the number of infeasible indexes in MKR decreases monotonically in the outermost layer, its convergence rate is lower than that of RAS because the safeguard strategies are time-consuming.
	
	In MATLAB quadprog routine, there is another algorithm which is called trust-region-reflective algorithm.  Since it has been shown in \cite{hungerlander2015feasible} that MKR outperforms trust-region-reflective algorithm in the same experiment, we do not add this method for comparison here. To sum up, in this benchmark set, RAS fully shows its robustness to condition number and density, and is worthy of being the best of these tested algorithms.
	
	\begin{figure}
		\centering
		\includegraphics[scale=0.8]{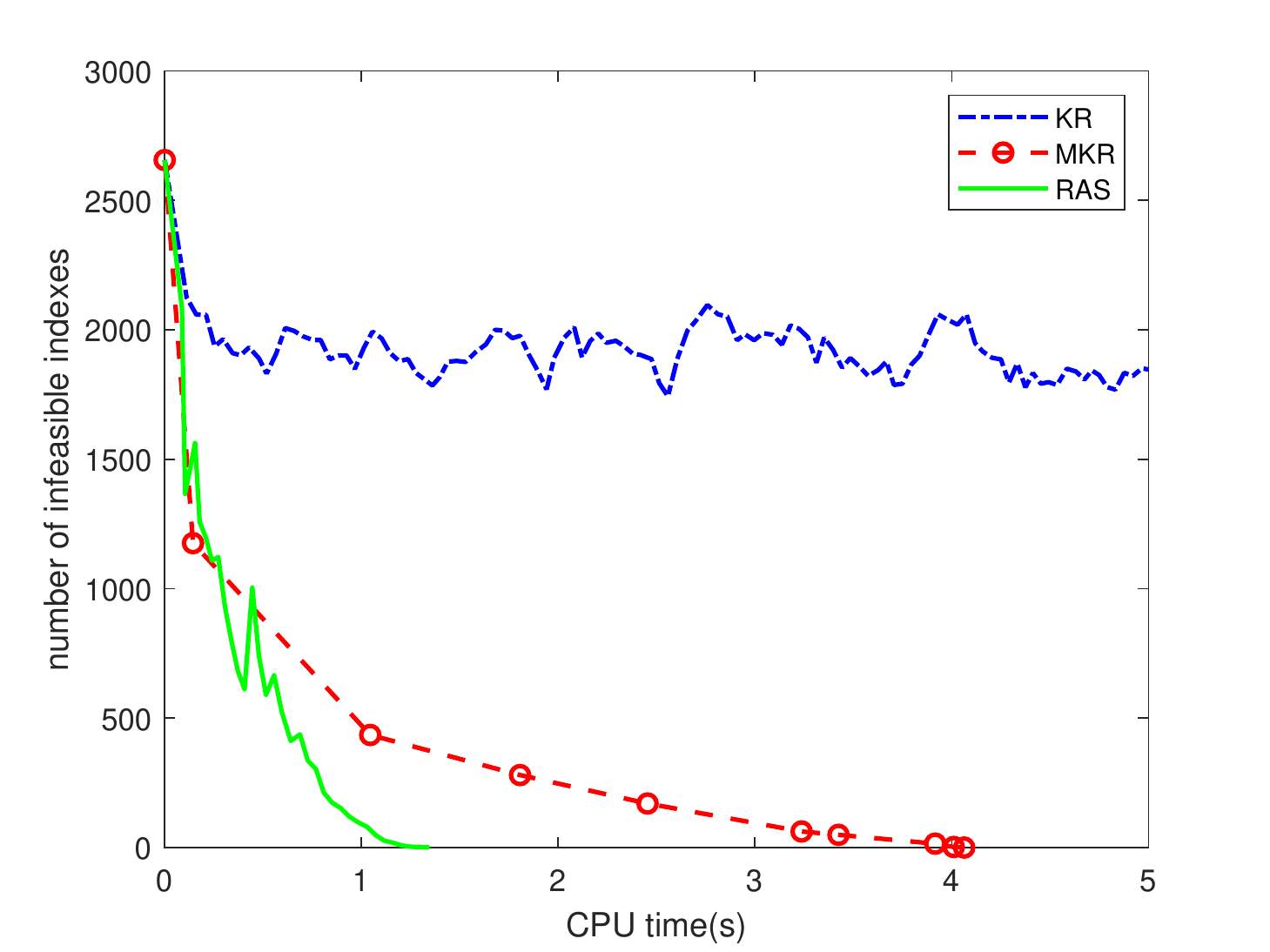}\\
		\caption{The number of infeasible indexes in KR, MKR and RAS changes with CPU time. (Medium case with $n = 5000$, $den=0.01$ and $\epsilon=1e-10$)}\label{fig:medium}
	\end{figure}
	
	\begin{sidewaystable}[!htb]
		\centering
		\setlength\tabcolsep{3pt}
		\begin{tabular}{|r|r|r|r|r|r|r|r|r|r|r|r|r|r|r|r|}
			\hline
			n & dens & cond & \multicolumn{4}{c|}{kr} & \multicolumn{3}{c|}{mkr} & \multicolumn{3}{c|}{ras} & \multicolumn{3}{c|}{quadprog} \\
			&       &       & time & solve & avgI & fail & time & solve & avgI & time & solve & avgI & time & iter & error \\
			\hline
			1000 & 0.1 & 1.00E+02 & 0.041 & 6.3 & 450.122 & 0 & 0.064 & 8.8 & 456.889 & 0.049 & 8.6 & 457.676 & 0.901 & 8.8 & 9.28E-09 \\
			1000 & 0.1 & 1.00E+06 & 0.164 & 29.5 & 376.756 & 2 & 0.180 & 45.4 & 322.796 & 0.062 & 17.5 & 359.532 & 1.327 & 9.2 & 5.94E-08 \\
			1000 & 0.1 & 1.00E+10 & nan & nan & nan & 10 & 0.334 & 98.7 & 300.991 & 0.099 & 29.5 & 338.233 & 1.396 & 9.9 & 1.92E-06 \\
			1000 & 0.1 & 1.00E+14 & nan & nan & nan & 10 & 0.503 & 172.8 & 279.904 & 0.140 & 49.8 & 317.139 & 1.366 & 9.7 & 3.84E-06 \\
			\hline
			1000 & 0.01 & 1.00E+02 & 0.011 & 6.0 & 501.587 & 0 & 0.016 & 7.7 & 505.461 & 0.014 & 8.0 & 501.137 & 0.244 & 8.3 & 2.56E-07 \\
			1000 & 0.01 & 1.00E+06 & 0.028 & 16.4 & 461.397 & 3 & 0.034 & 21.1 & 436.079 & 0.020 & 14.1 & 436.711 & 0.283 & 9.6 & 2.39E-06 \\
			1000 & 0.01 & 1.00E+10 & 0.037 & 28.6 & 445.153 & 5 & 0.037 & 27.4 & 410.633 & 0.017 & 15.4 & 406.739 & 0.284 & 10.0 & 5.81E-06 \\
			1000 & 0.01 & 1.00E+14 & nan & nan & nan & 10 & 0.071 & 47.6 & 417.596 & 0.017 & 19.5 & 359.511 & 0.300 & 10.1 & 1.03E-05 \\
			\hline
			5000 & 0.1 & 1.00E+02 & 1.933 & 6.1 & 2247.106 & 0 & 3.412 & 10.5 & 2334.588 & 2.431 & 8.6 & 2307.621 & 72.167 & 9.2 & 1.93E-08 \\
			5000 & 0.1 & 1.00E+06 & 4.477 & 18.3 & 1856.003 & 0 & 9.973 & 66.3 & 1603.614 & 3.399 & 20.3 & 1839.807 & 145.957 & 10.0 & 3.45E-07 \\
			5000 & 0.1 & 1.00E+10 & 12.274 & 55.2 & 1642.681 & 0 & 18.894 & 182.2 & 1364.746 & 3.942 & 29.2 & 1670.412 & 204.534 & 10.9 & 7.63E-06 \\
			5000 & 0.1 & 1.00E+14 & nan & nan & nan & 10 & 35.128 & 413.8 & 1258.065 & 5.156 & 42.3 & 1569.568 & 221.607 & 10.8 & 1.82E-05 \\
			\hline
			5000 & 0.01 & 1.00E+02 & 0.966 & 6.3 & 2389.914 & 0 & 1.718 & 10.5 & 2488.564 & 1.462 & 9.3 & 2467.775 & 43.785 & 8.6 & 2.75E-07 \\
			5000 & 0.01 & 1.00E+06 & 7.190 & 96.1 & 2001.940 & 3 & 3.465 & 59.5 & 1771.884 & 1.285 & 19.1 & 1891.670 & 75.881 & 9.5 & 7.61E-06 \\
			5000 & 0.01 & 1.00E+10 & nan & nan & nan & 10 & 5.874 & 124.0 & 1636.032 & 2.064 & 36.9 & 1789.822 & 101.886 & 11.1 & 1.88E-05 \\
			5000 & 0.01 & 1.00E+14 & nan & nan & nan & 10 & 9.001 & 222.1 & 1548.105 & 2.790 & 61.3 & 1676.945 & 109.446 & 11.0 & 6.09E-05 \\
			\hline
			5000 & 0.001 & 1.00E+02 & 0.038 & 5.8 & 2724.865 & 0 & 0.061 & 7.9 & 2639.070 & 0.060 & 8.9 & 2637.268 & 2.514 & 7.7 & 9.43E-06 \\
			5000 & 0.001 & 1.00E+06 & 0.057 & 10.8 & 2624.934 & 6 & 0.108 & 17.9 & 2428.910 & 0.063 & 12.5 & 2383.115 & 2.945 & 9.7 & 1.54E-05 \\
			5000 & 0.001 & 1.00E+10 & nan & nan & nan & 10 & 0.119 & 22.2 & 2388.820 & 0.063 & 14.1 & 2234.562 & 3.295 & 11.2 & 3.96E-05 \\
			5000 & 0.001 & 1.00E+14 & nan & nan & nan & 10 & 0.153 & 27.6 & 2360.347 & 0.054 & 17.5 & 1765.256 & 3.297 & 10.5 & 3.63E-05 \\
			\hline
			10000 & 0.001 & 1.00E+02 & 0.594 & 6.2 & 5190.824 & 0 & 1.014 & 9.3 & 5366.164 & 0.883 & 9.2 & 5309.676 & 47.977 & 8.0 & 7.88E-06 \\
			10000 & 0.001 & 1.00E+06 & nan & nan & nan & 10 & 1.603 & 35.4 & 4569.452 & 0.760 & 17.3 & 4519.208 & 90.036 & 10.1 & 3.25E-05 \\
			10000 & 0.001 & 1.00E+10 & nan & nan & nan & 10 & 2.244 & 62.6 & 4305.715 & 0.837 & 23.8 & 4276.534 & 116.439 & 12.0 & 3.41E-05 \\
			10000 & 0.001 & 1.00E+14 & nan & nan & nan & 10 & 2.550 & 71.1 & 4229.668 & 0.898 & 30.8 & 3678.914 & 133.709 & 11.6 & 1.06E-04 \\
			\hline
		\end{tabular}
		\caption{Random problems generated by sprandsym. Each line is the average over 10 trials.} \label{tab:t2}
	\end{sidewaystable}

	\subsection{Hard case}
	
	As we found that the density of $ Q $ greatly affects the performance of active set methods, here in this subsection, we test the performance of each algorithm when $  Q $ is a dense matrix. The combination of large condition number and high density will bring great difficulties to the algorithms, which is called hard case. We generate the random problem through the following MATLAB code:
	$$
	g = rand(n,1)-0.5;\, O=randn(n);\,O=qr(O);
	$$
	$$
	D=diag(cond.\hat{ }((0:n-1)/(n-1)));\, Q=O*D*O';
	$$
	where $ O $ is a random orthogonal matrix. Here the eigenvalue of $ Q $ is proportional and the condition number is noted as cond. The tolerance on dual nonnegativity violation of the three active set algorithms is set to $tol=1e-10$.
	
	The comparison of the algorithms in this benchmark set is shown in \cref{tab:t3}. The three active set algorithms are all sensitive to the condition number and  insensitive to the dimension, which is consistent with the test in the previous subsection. Since here we use dense $ Q $ instead, the number of solving linear equations for active set methods should be large according to the observations in the previous subsection. In fact, that is true because there are many examples that KR cannot solve, and the number of solving linear equations in MKR has reached more than 1000, but interestingly, the number of solving linear equations in RAS can still be stable within 50. This leads to the fact that in some cases with large condition number, the amount of computation of RAS is much smaller than that of MKR. For example, when $n = 2000$, $cond = 1e+14$,  the number of solving linear equations in MKR is 28 times that in RAS, and the computational time is more than 15 times that of RAS. Undoubtedly, RAS is the best of the three active set algorithms for the hard case.
	
	The  interior point method in quadprog routine is also involved in this comparison. We still set its termination criteria as the default. In some cases, it terminates in advance when the accuracy is not enough, which can be seen from the error column in \cref{tab:t3}. So the corresponding computational time is much less than that when the termination criteria for accuracy are met. Even with such a large error, in almost all cases,  the computational time of quadprog is still longer than that of RAS. This also helps to reflect the extraordinary performance of RAS.

	\begin{sidewaystable}[!htb]
		\centering
		\setlength\tabcolsep{3pt}
		\begin{tabular}{|r|r|r|r|r|r|r|r|r|r|r|r|r|r|r|}
			\hline
			n & cond & \multicolumn{4}{c|}{kr} & \multicolumn{3}{c|}{mkr} & \multicolumn{3}{c|}{ras} & \multicolumn{3}{c|}{quadprog} \\
			&       & time & solve & avgI & fail & time & solve & avgI & time & solve & avgI & time & iter & error \\
			\hline
			500 & 1.00E+06 & 0.021 & 13.6 & 231.275 & 0 & 0.070 & 46.6 & 199.865 & 0.017 & 16.9 & 220.754 & 0.044 & 10.9 & 1.04E-06 \\
			500 & 1.00E+10 & 0.183 & 110.7 & 244.689 & 7 & 0.228 & 211.3 & 179.136 & 0.028 & 26.1 & 225.733 & 0.044 & 11.4 & 2.38E-03 \\
			500 & 1.00E+14 & nan & nan & nan & 10 & 0.837 & 972.3 & 155.839 & 0.057 & 47.9 & 228.239 & 0.048 & 11.4 & 8.19E-01 \\
			\hline
			1000 & 1.00E+06 & 0.122 & 13.5 & 460.741 & 0 & 0.369 & 56.4 & 391.470 & 0.111 & 17.3 & 439.651 & 0.284 & 11.0 & 1.45E-06 \\
			1000 & 1.00E+10 & 1.789 & 195.0 & 496.405 & 9 & 1.104 & 237.8 & 356.478 & 0.150 & 26.1 & 447.151 & 0.302 & 12.1 & 2.05E-03 \\
			1000 & 1.00E+14 & nan & nan & nan & 10 & 3.818 & 1105.8 & 313.282 & 0.268 & 44.7 & 455.374 & 0.312 & 12.0 & 4.95E-01 \\
			\hline
			2000 & 1.00E+06 & 0.559 & 14.5 & 931.086 & 0 & 1.652 & 61.8 & 797.758 & 0.457 & 18.2 & 887.930 & 1.594 & 11.1 & 2.23E-06 \\
			2000 & 1.00E+10 & nan & nan & nan & 10 & 4.970 & 257.3 & 716.554 & 0.690 & 28.2 & 901.357 & 1.735 & 12.3 & 2.63E-03 \\
			2000 & 1.00E+14 & nan & nan & nan & 10 & 18.066 & 1264.4 & 632.899 & 1.155 & 45.2 & 921.583 & 1.815 & 12.5 & 4.46E-01 \\
			\hline
			4000 & 1.00E+06 & 2.443 & 14.4 & 1836.387 & 0 & 7.787 & 64.1 & 1576.270 & 2.415 & 19.0 & 1764.082 & 11.792 & 11.7 & 3.16E-06 \\
			4000 & 1.00E+10 & nan & nan & nan & 10 & 24.998 & 282.4 & 1426.014 & 3.584 & 28.4 & 1801.039 & 12.008 & 13.1 & 1.76E-03 \\
			4000 & 1.00E+14 & nan & nan & nan & 10 & 92.300 & 1377.6 & 1250.603 & 6.089 & 45.3 & 1814.294 & 12.822 & 13.1 & 3.44E-01 \\
			\hline
		\end{tabular}%
		\caption{Random problems with dense Q. Each line is the average over 10 trials.} \label{tab:t3}
	\end{sidewaystable}

	\section{Conclusions}
	\label{sec:5}
	
	In this paper, we considered the problem of SCQP with simple bound constraints. We proposed a random active set method, called RAS for short. This method has the characteristics of rapid adjustment of active set without getting stuck in a cycle. In the design of RAS, the classification of infeasible indexes helps to improve the robustness and efficiency of the algorithm. The finite step convergence was also proved in the sense of probability. In numerical experiments, we found that RAS is sensitive to the condition number and density of the matrix $ Q $, which is the same as other methods, but it is least affected. Numerical experiments also fully show the great efficiency and robustness of RAS.
	
	This is the first method that combines randomness and active set method. The idea is simple but powerful. It is not only applicable to quadratic programming, but also can be extended to other problems.
	
	Meanwhile, the method we used to prove the finite convergence may not be the best, so we only get the basic convergence result. If better mathematical tools are used, better results may be obtained, including more theoretical results for the design of those adjustable probabilities.
	
	%

	\section*{Acknowledgments}
	This manuscript is completed at Nankai University, with the support from School of Statistics and Data Science, School of Mathematics Science, and NSFC grant 12001297.

	\bibliographystyle{siamplain}
	\bibliography{references}
\end{document}

%% file: ex_shared.tex
\usepackage{lipsum}
\usepackage{amsfonts}
\usepackage{graphicx}
\usepackage{epstopdf}
\usepackage{algorithmic}
\usepackage{multirow,multicol}
\ifpdf
  \DeclareGraphicsExtensions{.eps,.pdf,.png,.jpg}
\else
  \DeclareGraphicsExtensions{.eps}
\fi

\newsiamremark{remark}{Remark}
\newsiamremark{hypothesis}{Hypothesis}
\crefname{hypothesis}{Hypothesis}{Hypotheses}
\newsiamthm{claim}{Claim}

\headers{A Random Active Set Method for SCQP with Simple Bounds}{Ran Gu, and Bing Gao}

\title{A Random Active Set Method for Strictly Convex Quadratic Problem with Simple Bounds\thanks{Submitted to the editors DATE.
\funding{This work was funded in part by NSFC grant 12001297 }}}

\author{Ran Gu\thanks{School of Statistics and Data Science, Nankai University, Tianjin 300071, P.R. China
  (\email{rgu@nankai.edu.cn}).}
\and Bing Gao\thanks{School of Mathematical Sciences and LPMC, Nankai University, Tianjin 300071, P.R. China,
  (\email{gaobing@nankai.edu.cn}).}
}

\usepackage{amsopn}
